\documentclass[12pt]{amsart}
\usepackage{amssymb,amsmath,amsthm,mathrsfs,multirow,xcolor,framed,url}
\usepackage[pdftex,
         pdfauthor={Dandan Chen, Rong Chen and Siyu Yin},
         pdftitle={On the total number of ones associated with cranks of partitions modulo 11},
         pdfsubject={MATHEMATICS},
         pdfkeywords={Total number of ones, Crank for partitions, Theta-functions.},
         pdfproducer={Latex with hyperref},
         pdfcreator={pdflatex}]{hyperref}
\usepackage [latin1]{inputenc}
\newtheorem{theorem}{Theorem}[section]
\newtheorem{lemma}[theorem]{Lemma}
\newtheorem{cor}[theorem]{Corollary}

\theoremstyle{definition}
\newtheorem{definition}[theorem]{Definition}

\theoremstyle{remark}

\numberwithin{equation}{section}

\newcommand\nutwid{\overset {\text{\lower 3pt\hbox{$\sim$}}}\nu}

\allowdisplaybreaks

\newcommand\omycite[1]{}

\newcommand{\beqs}{\begin{equation*}}
\newcommand{\eeqs}{\end{equation*}}
\newcommand{\beq}{\begin{equation}}
\newcommand{\eeq}{\end{equation}}
\renewcommand{\MR}[1]{\href{http://www.ams.org/mathscinet-getitem?mr={#1}}{MR{#1}}}

\begin{document}
\title[Congruences modulo powers of $7$ for $k$-elongated plane partitions]{Congruences modulo powers of $7$ for $k$-elongated plane partitions }


\author{Dandan Chen}
\address{Department of Mathematics, Shanghai University, People's Republic of China}
\address{Newtouch Center for Mathematics of Shanghai University, Shanghai, People's Republic of China}
\email{mathcdd@shu.edu.cn}
\author{Tianjian Xu}
\address{Department of Mathematics, Shanghai University, People's Republic of China}
\email{mathxtj@163.com}
\author{Siyu Yin*}
\address{Department of Mathematics, Shanghai University, People's Republic of China}
\email{siyuyin@shu.edu.cn, siyuyin0113@126.com}



\date{}


\begin{abstract}
The enumeration $d_k(n)$ of $k$-elongated plane partition diamonds has emerged as a generalization of the classical integer partition function $p(n)$. Congruences for $d_k(n)$ modulo certain powers of primes have been proven via elementary means and modular forms by many authors. Recently, Banerjee and Smoot established an infinite family of congruences for $d_5(n)$ modulo powers of 5. In this paper we have discovered an infinite congruence family for $d_3(n)$ and $d_5(n)$ modulo powers of 7.
\end{abstract}

\maketitle


\section{Introduction}
A partition of a positive integer  $n$ is defined as a sequence of positive integers in non-increasing order that sum to $n$. Denoted by $p(n)$, it represents the count of such partitions for $n$. The following distinguished congruences were discovered and proved by Ramanujan in \cite{Ramanujan-1919}:
\begin{align*}
p(5n+4)\equiv 0 &\pmod 5,\\
p(7n+5)\equiv 0 &\pmod 7,\\
p(11n+6)\equiv 0 &\pmod {11},
\end{align*}
for every non-negative integer $n$. In fact, for $\alpha\geq1$,
\begin{align}
p(5^{\alpha}n+\delta_{5,\alpha})&\equiv 0\pmod {5^{\alpha}},\label{5-congruence}\\
p(7^{\alpha}n+\delta_{7,\alpha})&\equiv 0\pmod {7^{[\frac{\alpha+2}{2}]}},\label{7-congruence}\\
p(11^{\alpha}n+\delta_{11,\alpha})&\equiv 0\pmod {11^{\alpha}}\label{11-congruence},
\end{align}
where $\delta_{t,\alpha}$ is the reciprocal modulo $t^{\alpha}$ of 24. \eqref{5-congruence} and \eqref{7-congruence} were first proved by Watson \cite{Watson-1938} in 1938. For an elementary of \eqref{5-congruence} see Hirschhorn and Hunt \cite{Hirschhorn-1981} and for an elementary proof of \eqref{7-congruence} see Garvan \cite{Garvan-1984}. \eqref{11-congruence} was proved by Atkin \cite{Atkin-1967} in 1967.

Over the years, Andrews and Paule \cite{Andrews-Paule-2022}  published an extremely influential series of papers developing an algorithmic methodology on MacMahon's partition analysis. Among the many contributions of this series of papers has been the development of the $k$-elongated plane partition function $d_k(n)$, which is enumerated by the following generating function
\begin{align}\label{Generating-function}
D_k(q):=\sum_{n=0}^{\infty}d_k(n)q^n=\prod_{n=0}^{\infty}\frac{(1-q^{2n})^k}{(1-q^n)^{3k+1}}.
\end{align}
The number $d_k(n)$ of partitions is found by summing the links of the $k$-elongated partition diamonds of length $n$. They then discovered several congruences for $d_1$, $d_2$, and $d_3$ modulo certain powers of primes, primarily using the Mathematica package RaduRK implemented by Smoot \cite{Smoot-2021}, which is based on the Ramanujan-Kolberg algorithm presented by Radu \cite{Radu-2015}.

Since the inception of the function $d_k(n)$, various authors have extended its congruence properties through elementary $q$-series manipulations and modular forms. da Silva, Hirschhorn, and Sellers \cite{Silva-Hirschhorn-Sellers-2022} gave elementary proofs of several congruences for $d_k(n)$ found by Andrews and Paule \cite{Andrews-Paule-2022} and added new individual congruences modulo small prime powers. Yao \cite{Yao-2024} provided elementary proofs of the congruences modulo 81, 243, and 729 for $d_k(n)$ conjectured by Andrews and Paule \cite[Conjectures 1 and 2]{Andrews-Paule-2022}. Baruah, Das, and Talukdar \cite{Baruah-Das_Talukdar-2023} found infinite families of congruences for $d_k(n)$ modulo powers of 2 and 3 and proved the following refinement of the result of da Silva, Hirschhorn, and Sellers \cite{Silva-Hirschhorn-Sellers-2022} on the existence of infinite congruence families for $d_k(n)$ modulo prime powers.
Besides, utilizing the localization method, Banerjee and Smoot \cite{Banerjee-Smoot-2025} gave the congruence results of $d_5(n)$ by various powers of 5.

In this paper, we focus on the  divisibility of 7 for $d_3(n)$ and $d_5(n)$. Following the method in \cite{Chen-Chen-Garvan-2024}, we will prove the following theorems:
\begin{theorem}\label{D3-7-power}\cite[Conjecture 1.4]{Banerjee-Smoot-2025-7}
For all $\alpha\geq1$ and all $n\geq0$ we have
\begin{align}\label{D3-7-formula}
d_3(7^{\alpha}n+\lambda_{\alpha})\equiv0\pmod{7^{\lfloor\frac{\alpha}{2}\rfloor}},
\end{align}
where $\lambda_{\alpha}=\frac{5\cdot7^{\alpha}+1}{6}$.
\end{theorem}

\begin{theorem}\label{D5-7-power}
For all $\alpha\geq1$ and all $n\geq0$ we have
\begin{align}\label{D5-7-formula}
d_5(7^{\alpha+2}n+\lambda_{\alpha+2})d_5(\lambda_{\alpha})\equiv d_5(7^{\alpha}n+\lambda_{\alpha})d_5(\lambda_{\alpha+2}) \pmod{7^{\lfloor\frac{\alpha}{2}\rfloor+2}},
\end{align}
where
\begin{align*}
\lambda_{\alpha}=
\begin{cases}
    \frac{1+7^{\alpha}}{4},&\text{if}~\alpha~\text{is odd},\\
    \frac{1+3\cdot7^{\alpha}}{4},&\text{if}~ \alpha~\text{is even}.
\end{cases}
\end{align*}
\end{theorem}

Besides, utilizing the generating function of $D_k$, we have
\begin{cor}\label{k-modulo-7}
For $k\geq0$, we have
\begin{align}
\label{7k-7}d_{7k}(7n+5)\equiv0\pmod7,\\
\label{7k+5}d_{7k+5}(7n+2)\equiv0\pmod7.
\end{align}
\end{cor}

\begin{proof}
We extract the items of $q^{7n+5}$ in the formula below
\begin{align*}
D_{7k}=\sum_{n=0}^{\infty}d_{7k}(n)q^n=\frac{J_2^{7k}}{J_1^{21k+1}}\equiv\frac{J_{14}^k}{J_7^{3k}J_1}\pmod 7,
\end{align*}
recalling the fact that $p(7n+5)\equiv0\pmod 7$, we directly get \eqref{7k-7}. Similarly, \eqref{7k+5} can be proved by \eqref{L1}.
\end{proof}

The paper is organized as follows. In Section \ref{sec-pre} we introduce basic notations, the necessary background, and algorithms from the theory of modular functions. In Section \ref{sec-lemma} we derive modular equations on $\Gamma_0(14)$. In Sections \ref{sec-D3} and \ref{sec-D5} we apply the theory of modular functions to prove our Theorems \ref{D3-7-power} and \ref{D5-7-power}.

\section{Preliminaries}\label{sec-pre}
In this paper, we use the following conventions: $\mathbb{N}=\{0,1,2,\cdots\}$ denotes the set of nonnegative integers. The complex upper half plane is denoted by $\mathbb{H}:=\{\tau\in\mathbb{C}:\Im(\tau)\textgreater 0\}$. We will also use the short hand notation for the Dedekind eta function:
\begin{align*}
\eta_n(\tau):=\eta(n\tau)=q^{\frac{n}{24}}\prod_{k=1}^\infty(1-q^{nk}),~~n\in\mathbb{Z},~\tau\in\mathbb{H}.
\end{align*}
For $\theta$-products we use $J_b:=\prod_{k=1}^\infty(1-q^{kb})$. Throughout this paper, for $x\in\mathbb{R}$ the symbol $\lfloor x\rfloor$ denotes the largest integer less than or equal to $x$, and $\lceil x\rceil$ denotes the smallest integer greater than or equal to $x$.

Let $f=\sum_{n\in\mathbb{Z}}a_nq^n$, $f\neq 0$, be such that $a_n=0$ for almost all $n\textless 0$. We define the order of $f$ (with respect to $q$) as the smallest integer $N$ such that $a_N\neq0$ and write $N=ord_q(f)$. For instance, we let $F=f\circ t=\sum_{n\in\mathbb{Z}}a_nt^n$ where $t=\sum_{n\geq 1}b_nq^n$, then the $t$-order of $F$ is defined to be the smallest integer $N$ such that $a_N\neq 0$ and write $N=ord_t(F)$.

\begin{definition}
For $f:\mathbb{H}\rightarrow\mathbb{C}$ and $m\in \mathbb{N}^{\ast}$ we define $U_m(f):\mathbb{H} \rightarrow \mathbb{C}$ by
\begin{align*}
U_m(f)(\tau):=\frac{1}{m}\sum_{\lambda=0}^{m-1}f(\frac{\tau+\lambda}{m}),~\tau\in\mathbb{H}.
\end{align*}
\end{definition}
$U_m$ is linear (over $\mathbb{C}$); in addition, it is easy to verify that
\begin{align*}
U_{mn}=U_m\circ U_n=U_n\circ U_m,~m,n\in\mathbb{N}^{\ast}.
\end{align*}
The operators $U_m$, introduced by Atkin and Lehner \cite{Atkin-Lehner-1970}, are closely related to Hecke operators. They typically arise in the context of partition congruences \cite{Andrews-1976} mostly because of the property: if
\begin{align*}
f(\tau)=\sum_{n=-\infty}^{\infty}f_nq^n~~(q=e^{2\pi i\tau}),
\end{align*}
then
\begin{align*}
U_m(f)(\tau)=\sum_{n=-\infty}^{\infty}f_{mn}q^n.
\end{align*}

\begin{definition}\label{t-p-y}
Let $t, p_0, p_1$ be functions defined on $\mathbb{H}$ as follows:
\begin{align*}
t:=\frac{\eta_{7}^4}{\eta_1^4},~~~p_0:=\frac{1}{7}(8\frac{\eta_2^7\eta_7}{\eta_1^7\eta_{14}}-1),~~~p_1:=\frac{\eta_{2}^4\eta_7^{4}}{\eta_1^4\eta_{14}^{4}},
\end{align*}
 which have Laurent series expansions in powers of $q$ with coefficients in $\mathbb{Z}$.
\end{definition}
Remark: Since $(1-q)^7\equiv1-q^7\pmod 7$, we have $\frac{J_{7k}}{J_k^7}\equiv1\pmod7$, which indicates $p_0$ have Laurent series expansions in powers of $q$ with coefficients in $\mathbb{Z}$.

\begin{definition}
A map $a:\mathbb{Z}\rightarrow\mathbb{Z}$ is called a discrete  function if it has finite support. A map $a:\mathbb{Z}\times\mathbb{Z}\rightarrow\mathbb{Z}$ is called discrete array if for each $i\in\mathbb{Z}$ the map $a(i,-):\mathbb{Z}\rightarrow\mathbb{Z},k\mapsto a(i,,k)$, has finite support.
\end{definition}

\section{Modular equation}
\label{sec-lemma}

Our proof relies on the identities in the Appendix. All these identities can be proved using modular function theory.
Regarding to the valence formula in the modular forms, Garvan has written a MAPLE package called ETA which can prove Group I--IX in the Appendix automatically. See
\begin{align}
\label{r:eta}
https://qseries.org/fgarvan/qmaple/ETA/.
\end{align}
The tutorial for this see \cite{gtutorial}.

\begin{theorem}\cite[Theorem 2.6]{Chen-Chen-Garvan-2024}\label{def-aj-t}
With $t=t(\tau)$ as in Definition \ref{t-p-y} we let
\begin{align}
\label{t-a0}a_0(t)=&t;\\
\label{t-a1}a_1(t)=&7^2t^2+4\cdot7t;\\
\label{t-a2}a_2(t)=&7^4t^3+4\cdot7^3t^2+46\cdot7t;\\
\label{t-a3}a_3(t)=&7^6t^4+4\cdot7^5t^3+46\cdot7^3t^2+272\cdot7t;\\
\label{t-a4}a_4(t)=&7^8t^5+4\cdot7^7t^4+46\cdot7^5t^3+272\cdot7^3t^2+845\cdot7t;\\
\label{t-a5}a_5(t)=&7^{10}t^6+4\cdot7^9t^5+46\cdot7^7t^4+272\cdot7^5t^3+845\cdot7^3t^2+176\cdot7^2t;\\
\label{t-a6}a_6(t)=&7^{12}t^7+4\cdot7^{11}t^6+46\cdot7^9t^5+272\cdot7^7t^4+845\cdot7^5t^3+176\cdot7^4t^2+82\cdot7^2t.
\end{align}
Then
\begin{align}\label{modular-identity}
t(\tau)^7-\sum_{l=0}^6a_l(t(7\tau))t(\tau)^l=0.
\end{align}
\end{theorem}
We define $s:\{0,1,\cdots,6\}\times\{1,2,\cdots,7\}\rightarrow\mathbb{Z}$ to be the unique function satisfying
\begin{align}\label{def-aj}
a_j(t)=\sum_{l=1}^7 s(j,l)7^{\lfloor\frac{7l+j-4}{4}\rfloor}t^l.
\end{align}

\begin{lemma}\label{U2-j-jl2}
For $u:\mathbb{H}\rightarrow\mathbb{C}$ and $j\in\mathbb{Z}$, then
\begin{align}
\label{re-t}&U_7(ut^j)=\sum_{l=0}^6a_l(\tau)U_7(ut^{j+l-7}),
\end{align}
where $t=t(\tau)$ is defined in Definition \ref{t-p-y} and $a_j(\tau)$ are given in \eqref{t-a0}--\eqref{t-a6}.
\end{lemma}

\begin{proof}
The results follows easily from \eqref{modular-identity} by multiplying both sides by $ut^{j-7}$ and applying the operator $U_7$.
\end{proof}

\begin{lemma}\label{U7-order}
Let $u,v_1,v_2,v_3:\mathbb{H}\rightarrow\mathbb{C}$ and $l\in\mathbb{Z}$. Suppose for $l\leq k\leq l+6$,  there exist polynomials $p_k^{(i)}(t)\in\mathbb{Z}[t,t^{-1}]$ such that
\begin{align}\label{ord-1}
U_7(ut^k)=v_1p_k^{(1)}(t)+v_2p_k^{(2)}(t)+v_3p_k^{(3)}(t)
\end{align}
and
\begin{align}\label{ord-2}
ord_t(p_k^{(i)}(t))\geq \lceil\frac{k+s_i}{7}\rceil,
\end{align}
for a fixed integer $s$. Then there exist families of polynomials $p_k^{(i)}(t)\in\mathbb{Z}[t,t^{-1}]$, integer $k\in\mathbb{Z}$, such that \eqref{ord-1} and \eqref{ord-2} hold for all $k\in\mathbb{Z}$.
\end{lemma}

\begin{proof}
Let $N\textgreater l+6$ be an integer and assume by induction that there are families of polynomials $p_k(t)$ such that \eqref{ord-1} and \eqref{ord-2} hold for $l\leq k\leq N-1$. Suppose
\begin{align*}
p_k^{(i)}(t)=\sum_{n\geq\lceil\frac{k+s_i}{7}\rceil}c_i(k,n)t^n, ~~~l\leq k\leq N-1,
\end{align*}
with integers $c_i(k,n)$. Applying Lemma \ref{U2-j-jl2} we get:
\begin{align*}
U_7(ut^N)=&\sum_{j=0}^6a_j(t)U_7(ut^{N+j-7})\\
=&\sum_{j=0}^6a_j(t)\sum_{i=1}^3v_i\sum_{n\geq\lceil\frac{N+j-7+s_i}{7}\rceil}c_i(N+j-7,n)t^n\\
=&\sum_{i=1}^3v_i\sum_{j=0}^6a_j(t)t^{-1}\sum_{n\geq\lceil\frac{N+j+s_i}{7}\rceil}c_i(N+j-7,n-1)t^n.
\end{align*}
Recalling the fact that $a_j(t)t^{-1}$ for $0\leq j\leq6$ is a polynomial in $t$, this determines polynomials $p_N^{(i)}(t)$ with the desired properties. The induction proof for $N\textless l$ is analogous.
\end{proof}

\begin{lemma}\label{U-power}
Let $v_1, v_2, v_3, u:\mathbb{H}\rightarrow\mathbb{C}$ and $l\in\mathbb{Z}$. Suppose for $l\leq k\leq l+6$  there exist polynomials $p_k^{(i)}(t)\in\mathbb{Z}[t,t^{-1}]$, such that
\begin{align}\label{power-1}
U_7(ut^k)=v_1p_k^{(1)}(t)+v_2p_k^{(2)}(t)+v_3p_k^{(3)}(t),
\end{align}
where
\begin{align}\label{power-2}
p_k^{(i)}(t)=\sum_nc_i(k,n)7^{\lfloor\frac{7n-k+\gamma_i}{4}\rfloor}t^n,
\end{align}
with integers $\gamma_i$ and $c_i(k,n)$. Then there exist families of polynomials $p_k^{(i)}(t)\in\mathbb{Z}[t,t^{-1}]$, $k\in\mathbb{Z}$, of the form \eqref{power-2} for which property \eqref{power-1} holds for  all $k\in\mathbb{Z}$.
\end{lemma}

\begin{proof}
Suppose for an integer $N\textgreater l+6$ there are families of polynomials $p_k^{(i)}(t)$, $i\in\{1,2,3\}$, of the form \eqref{power-2} satisfying property \eqref{power-1} for $l\leq k\leq N-1$. We proceed by mathematical induction on $N$. Applying Lemma \ref{U2-j-jl2}  and using the induction base \eqref{power-1} and \eqref{power-2} we attain:
\begin{align*}
U_7(ut^N)=\sum_{j=0}^6a_j(t)\sum_{i=1}^3v_i\sum_nc_i(N+j-7,n)7^{\lfloor{\frac{7n-(N+j-7)+\gamma)i}{4}}\rfloor}t^n.
\end{align*}
Utilizing \eqref{def-aj}  this rewrites into:
\begin{align*}
U_7(ut^N)&=\sum_{j=0}^6\sum_{l=1}^7s(j,l)7^{\lfloor{\frac{7l+j-4}{4}}\rfloor}t^l
\sum_{i=1}^3v_i\sum_nc_i(N+j-7,n)7^{\lfloor{\frac{7n-(N+j-7)+\gamma_i}{4}}\rfloor}t^n\\
&=\sum_{i=1}^3v_i\sum_{j=0}^6\sum_{l=1}^7\sum_ns(j,l)c_i(N+j-7,n-l)7^{\lfloor{\frac{7(n-l)-(N+j-7)+\gamma_i}{4}}\rfloor+\lfloor{\frac{7l+j-4}{4}}\rfloor}t^n.
\end{align*}
The induction step is completed by simplifying the exponent of $7$ as follows:
\begin{align*}
&\lfloor{\frac{7(n-l)-(N+j-7)+\gamma_i}{4}}\rfloor+\lfloor{\frac{7l+j-4}{4}}\rfloor\\
&\geq \lfloor{\frac{7(n-l)-(N+j-7)+\gamma_i+7l+j-4-3}{4}}\rfloor\\
&\geq\lfloor{\frac{7n-N+\gamma_i}{4}}\rfloor.
\end{align*}
The induction proof for $N\textless l$ works analogously.
\end{proof}

\section{Proof of Theorem \ref{D3-7-power}}
\label{sec-D3}
The proof depends on the fundamental relations in the Appendix.
Let $k=3$ in \eqref{Generating-function}, we have the generating function that
\begin{align*}
D_3(q)=\frac{J_2^3}{J_1^{10}}.
\end{align*}
For $f:\mathbb{H}\rightarrow\mathbb{C}$ we define $U_{A_3}(f), U_B(f):\mathbb{H} \rightarrow \mathbb{C}$ by $U_{A_3}(f):=U_7(A_3f)$, and $U_B(f):=U_7(Bf)$, where
\begin{align*}
A_3:=\frac{\eta_2^3\eta_{49}^{10}}{\eta_1^{10}\eta_{98}^3}~~\text{and}~~ B:=1.
\end{align*}
Define $L_0:=1$, and for $\alpha\geq0$ define
\begin{align*}
L_{2\alpha+1}=U_{A_3}(L_{2\alpha}),~~~~~L_{2\alpha+2}=U_B(L_{2\alpha+1}).
\end{align*}
With the $U_7$ operator, it is easily to verify that for $\alpha\geq1$,
\begin{align}
\label{3-series-odd}L_{2\alpha-1}=&q^2\frac{J_7^{10}}{J_{14}^3}\sum_{n=0}^{\infty}d_3(7^{2\alpha-1} n+\lambda_{2\alpha-1})q^n,\\
\label{3-series-even}L_{2\alpha}=&q\frac{J_1^{10}}{J_2^3}\sum_{n=0}^{\infty}d_3(7^{2\alpha} n+\lambda_{2\alpha})q^n,
\end{align}
where
\begin{align*}
\lambda_{\alpha}=\frac{5\cdot7^{\alpha}+1}{6}.
\end{align*}
To prove Theorem \ref{D3-7-power}, following lemmas are needed.

\begin{lemma}\label{3-concrete-relations}
For $j=0,1,2$ there exist discrete functions of $n$, $a_{k,i}(n,j)$ and $b_{k,i}(n,j)$ such that
\begin{align}
\label{A3-p0t}U_{A_3}(p_0t^k)=&p_0\sum_{n\geq\lceil\frac{k+14}{7}\rceil}7^{\lfloor\frac{7n-k-12}{4}\rfloor}\cdot a_{k,0}(n,0)t^n+p_1\sum_{n\geq\lceil\frac{k+15}{7}\rceil}7^{\lfloor\frac{7n-k-15}{4}\rfloor}\cdot a_{k,0}(n,1)t^n\\ \nonumber
&+\sum_{n\geq\lceil\frac{k+18}{7}\rceil}7^{\lfloor\frac{7n-k-15}{4}\rfloor}\cdot a_{k,0}(n,2)t^n,\\
\label{A3-p1t}U_{A_3}(p_1t^k)=&p_0\sum_{n\geq\lceil\frac{k+15}{7}\rceil}7^{\lfloor\frac{7n-k-11}{4}\rfloor}\cdot a_{k,1}(n,0)t^n+p_1\sum_{n\geq\lceil\frac{k+14}{7}\rceil}7^{\lfloor\frac{7n-k-15}{4}\rfloor}\cdot a_{k,1}(n,1)t^n\\
\nonumber&+\sum_{n\geq\lceil\frac{k+18}{7}\rceil}7^{\lfloor\frac{7n-k-14}{4}\rfloor}\cdot a_{k,1}(n,2)t^n,\\
\label{A3-t}U_{A_3}(t^k)=&p_0\sum_{n\geq\lceil\frac{k+14}{7}\rceil}7^{\lfloor\frac{7n-k-11}{4}\rfloor}\cdot a_{k,2}(n,0)t^n+p_1\sum_{n\geq\lceil\frac{k+15}{7}\rceil}7^{\lfloor\frac{7n-k-14}{4}\rfloor}\cdot a_{k,2}(n,1)t^n\\
\nonumber&+\sum_{n\geq\lceil\frac{k+17}{7}\rceil}7^{\lfloor\frac{7n-k-15}{4}\rfloor}\cdot a_{k,2}(n,2)t^n,\\
\label{B-p0t}U_{B}(p_0t^k)=&p_0\sum_{n\geq\lceil\frac{k-3}{7}\rceil}7^{\lfloor\frac{7n-k+2}{4}\rfloor}\cdot b_{k,0}(n,0)t^n+p_1\sum_{n\geq\lceil\frac{k+4}{7}\rceil}7^{\lfloor\frac{7n-k-1}{4}\rfloor}\cdot b_{k,0}(n,1)t^n\\
\nonumber&+\sum_{n\geq\lceil\frac{k+1}{7}\rceil}7^{\lfloor\frac{7n-k-1}{4}\rfloor}\cdot b_{k,0}(n,2)t^n,\\
\label{B-p1t}U_{B}(p_1t^k)=&p_0\sum_{n\geq\lceil\frac{k-3}{7}\rceil}7^{\lfloor\frac{7n-k+2}{4}\rfloor}\cdot b_{k,1}(n,0)t ^n+p_1\sum_{n\geq\lceil\frac{k+4}{7}\rceil}7^{\lfloor\frac{7n-k-1}{4}\rfloor}\cdot b_{k,1}(n,1)t^n\\
\nonumber&+\sum_{n\geq\lceil\frac{k+1}{7}\rceil}7^{\lfloor\frac{7n-k}{4}\rfloor}\cdot
b_{k,1}(n,2)t^n,\\
\label{B-t}U_{B}(t^k)=&\sum_{n\geq\lceil\frac{k}{7}\rceil}7^{\lfloor\frac{7n-k-1}{4}\rfloor}\cdot b_{k,2}(n,2)t^n.
\end{align}
\end{lemma}

\begin{proof}
The Appendix lists sixty-three fundamental relations. The seven fundamental relations of Group I fit the pattern of the relation \eqref{A3-p0t} for seven values of $k$. Applying Lemmas \ref{U7-order} and \ref{U-power} immediately proves the statement for all $k\in\mathbb{N}$. The other relations are proved by Group II--VI in the Appendix.
\end{proof}

\begin{lemma}\label{D3-L-alpha}
For integer $\alpha\geq1$ there exist integers $d_{n,i}^{(\alpha)},i=0,1,2$ such that
\begin{align}
\label{L-odd}L_{2\alpha-1}&=p_0\sum_{n\geq2}d_{n,0}^{(2\alpha-1)}7^{\lfloor\frac{7n-13}{4}\rfloor+\alpha-1}t^n+p_1\sum_{n\geq3}d_{n,1}^{(2\alpha-1)}7^{\lfloor\frac{7n-16}{4}\rfloor+\alpha-1}t^n+\sum_{n\geq3}d_{n,2}^{(2\alpha-1)}7^{\lfloor\frac{7n-16}{4}\rfloor+\alpha-1}t^n,\\
\label{L-even}L_{2\alpha}&=d_{0,0}^{(2\alpha)}7^{\alpha}p_0+p_0\sum_{n\geq1}d_{n,0}^{(2\alpha)}7^{\lfloor\frac{7n-4}{4}\rfloor+\alpha}t^n+p_1\sum_{n\geq1}d_{n,1}^{(2\alpha)}7^{\lfloor\frac{7n-7}{4}\rfloor+\alpha}t^n+\sum_{n\geq1}d_{n,2}^{(2\alpha)}7^{\lfloor\frac{7n-7}{4}\rfloor+\alpha}t^n.
\end{align}
\end{lemma}
\begin{proof}
By \eqref{A3-t} we see that $L_1$ has form given in \eqref{L-odd}. Assume that $L_{2\alpha-1}$ has the form given in \eqref{L-odd} for a fixed $\alpha$, then
\begin{align}\label{even-odd-1}
L_{2\alpha}&=U_B(L_{2\alpha-1})\\
\nonumber
&=\sum_{k\geq2}d_{k,0}^{(2\alpha-1)}7^{\lfloor\frac{7n-13}{4}\rfloor+\alpha-1}U_B(p_0t^k)+\sum_{k\geq3}d_{k,1}^{(2\alpha-1)}7^{\lfloor\frac{7n-16}{4}\rfloor+\alpha-1}U_B(p_1t^k)\\
\nonumber&+\sum_{k\geq3}d_{k,2}^{(2\alpha-1)}7^{\lfloor\frac{7n-16}{4}\rfloor+\alpha-1}U_B(t^k)   .
\end{align}
We aim to show that the form of each sum on the right-hand side of \eqref{even-odd-1} satisfies the form in \eqref{L-even}. From \eqref{B-p0t}, \eqref{B-p1t}, and \eqref{B-t}, we get
\begin{align*}
L_{2\alpha}=&p_0\sum_{k=2}^{\infty}\sum_{n\geq\lceil\frac{k-3}{7}\rceil}b_{k,0}(n,0)d_{k,0}^{(2\alpha-1)}7^{\lfloor\frac{7n-k+2}{4}\rfloor+\lfloor\frac{7k-13}{4}\rfloor+\alpha-1}t^n\\
&+p_1\sum_{k=2}^{\infty}\sum_{n\geq\lceil\frac{k+4}{7}\rceil}b_{k,0}(n,1)d_{k,0}^{(2\alpha-1)}7^{\lfloor\frac{7n-k-1}{4}\rfloor+\lfloor\frac{7k-13}{4}\rfloor+\alpha-1}t^n\\
&+\sum_{k=2}^{\infty}\sum_{n\geq\lceil\frac{k+1}{7}\rceil}b_{k,0}(n,2)d_{k,0}^{(2\alpha-1)}7^{\lfloor\frac{7n-k-1}{4}\rfloor+\lfloor\frac{7k-13}{4}\rfloor+\alpha-1}t^n\\
&+p_0\sum_{k=3}^{\infty}\sum_{n\geq\lceil\frac{k-3}{7}\rceil}b_{k,1}(n,0)d_{k,1}^{(2\alpha-1)}7^{\lfloor\frac{7n-k+2}{4}\rfloor+\lfloor\frac{7k-16}{4}\rfloor+\alpha-1}t^n\\
&+p_1\sum_{k=3}^{\infty}\sum_{n\geq\lceil\frac{k+4}{7}\rceil}b_{k,1}(n,1)d_{k,1}^{(2\alpha-1)}7^{\lfloor\frac{7n-k-1}{4}\rfloor+\lfloor\frac{7k-16}{4}\rfloor+\alpha-1}t^n\\
&+\sum_{k=3}^{\infty}\sum_{n\geq\lceil\frac{k+1}{7}\rceil}b_{k,1}(n,2)d_{k,1}^{(2\alpha-1)}7^{\lfloor\frac{7n-k}{4}\rfloor+\lfloor\frac{7k-16}{4}\rfloor+\alpha-1}t^n\\
&+\sum_{k=3}^{\infty}\sum_{n\geq\lceil\frac{k}{7}\rceil}b_{k,2}(n,2)d_{k,2}^{(2\alpha-1)}7^{\lfloor\frac{7n-k-1}{4}\rfloor+\lfloor\frac{7k-16}{4}\rfloor+\alpha-1}t^n.
\end{align*}
Since the proofs are similar we just consider the items of $p_0t^k$ in $L_{2\alpha}$. For $k\geq2$ and $n\geq\lceil\frac{k-3}{7}\rceil$,
\begin{align*}
\lfloor\frac{7n-k+2}{4}\rfloor+\lfloor\frac{7k-13}{4}\rfloor+\alpha-1\geq\lfloor\frac{7n-4}{4}\rfloor+\alpha.
\end{align*}
For $k\geq3$ and $n\geq\lceil\frac{k-3}{7}\rceil$,
\begin{align*}
\lfloor\frac{7n-k+2}{4}\rfloor+\lfloor\frac{7k-16}{4}\rfloor+\alpha-1\geq\lfloor\frac{7n-1}{4}\rfloor+\alpha.
\end{align*}
Then we focus on the coefficient of $p_0$ in $L_{2\alpha}$. Noticing that just $U_B(p_0t^2)$, $U_B(p_0t^3)$, and $U_B(p_1t^3)$ can have this item, while $p_0t^3$ and $p_1t^3$ have been satisfied the condition we need. Applying Lemma \ref{U2-j-jl2}, equation \eqref{B-p0t} and Group VII in the Appendix, we have
\begin{align*}
U_B(p_0t^2)=14p_0+7f(p_0,p_1,t).
\end{align*}
Thus, the coefficients of $p_0$ in $L_{2\alpha}$ can be divided by $7^{\alpha}$. So that the items of $p_0t^k$ in $L_{2\alpha}$ have the form of \eqref{L-even}. Similarly, the other items have the correct form. Hence $L_{2\alpha}$ has the desired form. The proof that the correct form of $L_{2\alpha}$ implies the correct form of $L_{2\alpha+1}$ is analogous. The general result follows by induction.
\end{proof}

\section{Proof of Theorem \ref{D5-7-power}}
\label{sec-D5}
Let $k=5$, we have the generating function that
\begin{align*}
D_5(q)=\frac{J_2^5}{J_1^{16}}.
\end{align*}
For $f:\mathbb{H}\rightarrow\mathbb{C}$ we define $U_{A_5}(f), U_B(f):\mathbb{H} \rightarrow \mathbb{C}$ by $U_{A_5}(f):=U_7(A_5f)$, and $U_B(f):=U_7(Bf)$, where
\begin{align*}
A_5:=\frac{\eta_2^5\eta_{49}^{16}}{\eta_1^{16}\eta_{98}^5}~~\text{and}~~ B:=1.
\end{align*}
Define $L_0=1$, and for $\alpha\geq0$ define
\begin{align}\label{5-L-definition}
L_{2\alpha+1}=U_{A_5}(L_{2\alpha}),~~~~~L_{2\alpha+2}=U_B(L_{2\alpha+1}).
\end{align}
Similar to  \eqref{3-series-odd} and \eqref{3-series-even}, we have for $\alpha\geq1$
\begin{align}\label{D5-L-alpha}
L_{2\alpha-1}=&q^2\frac{J_7^{16}}{J_{14}^5}\sum_{n=0}^{\infty}d_5(7^{2\alpha-1} n+\lambda_{2\alpha-1})q^n,\\ \nonumber
L_{2\alpha}=&q\frac{J_1^{16}}{J_2^5}\sum_{n=0}^{\infty}d_5(7^{2\alpha} n+\lambda_{2\alpha})q^n,
\end{align}
where
\begin{align*}
\lambda_{\alpha}=
\begin{cases}
    \frac{1+7^{\alpha}}{4},&if~\alpha~is~odd,\\
    \frac{1+3\cdot7^{\alpha}}{4},&if~ \alpha~is~even.
\end{cases}
\end{align*}
To prove Theorem \ref{D5-7-power}, following lemmas are needed.

\begin{lemma}\label{5-concrete-relations}
For $i,j=0,1,2$ there exist discrete functions of $n$, $c_{k,i}(n,j)$ such that the following relations hold for all $ k\in \mathbb{Z}$:
\begin{align}
\label{A5-p0}U_{A_5}(p_0t^k)=&p_0\sum_{n\geq\lceil\frac{k+19}{7}\rceil}7^{\lfloor\frac{7n-k-20}{4}\rfloor}\cdot c_{k,0}(n,0)t^n+p_1\sum_{n\geq\lceil\frac{k+19}{7}\rceil}7^{\lfloor\frac{7n-k-23}{4}\rfloor}\cdot c_{k,0}(n,1)t^n\\ \nonumber&+\sum_{n\geq\lceil\frac{k+26}{7}\rceil}7^{\lfloor\frac{7n-k-23}{4}\rfloor}\cdot c_{k,0}(n,2)t^n,\\
\label{A5-p1}U_{A_5}(p_1t^k)=&p_0\sum_{n\geq\lceil\frac{k+18}{7}\rceil}7^{\lfloor\frac{7n-k-19}{4}\rfloor}\cdot c_{k,1}(n,0)t^n+p_1\sum_{n\geq\lceil\frac{k+18}{7}\rceil}7^{\lfloor\frac{7n-k-23}{4}\rfloor}\cdot c_{k,1}(n,1)t^n\\ \nonumber&+\sum_{n\geq\lceil\frac{k+25}{7}\rceil}7^{\lfloor\frac{7n-k-22}{4}\rfloor}\cdot c_{k,1}(n,2)t^n,\\
\label{A5-1}U_{A_5}(t^k)=&p_0\sum_{n\geq\lceil\frac{k+19}{7}\rceil}7^{\lfloor\frac{7n-k-19}{4}\rfloor}\cdot c_{k,2}(n,0)t^n+p_1\sum_{n\geq\lceil\frac{k+19}{7}\rceil}7^{\lfloor\frac{7n-k-22}{4}\rfloor}\cdot c_{k,2}(n,1)t^n\\ \nonumber&+\sum_{n\geq\lceil\frac{k+26}{7}\rceil}7^{\lfloor\frac{7n-k-23}{4}\rfloor}\cdot c_{k,2}(n,2)t^n.
\end{align}
\end{lemma}
\begin{proof}
The Appendix lists sixty-three fundamental relations. The seven fundamental relations of Group IV fit the pattern of the relation \eqref{A5-p0} for three values of $k$. The same observations applies to the Groups V and VI with regard to the relations \eqref{A5-p1}, \eqref{A5-1}, respectively. Applying Lemmas \ref{U7-order} and \ref{U-power} immediately proves the statement for all $k\in \mathbb{Z}$.
\end{proof}
\begin{lemma}
Given the equalities of $U_B(p_0t^k)$ and $U_B(p_1t^k)$ in the Appendix, we have
\begin{align}\label{B-p0+p1}
U_{B}((p_0+p_1)t^k)=&p_0\sum_{n\geq\lceil\frac{k-1}{7}\rceil}7^{\lfloor\frac{7n-k+2}{4}\rfloor}\cdot \widetilde{b_k}(n,0)t^n+p_1\sum_{n\geq\lceil\frac{k+6}{7}\rceil}7^{\lfloor\frac{7n-k-1}{4}\rfloor}\cdot \widetilde{b_k}(n,1)t^n\\ \nonumber&+\sum_{n\geq\lceil\frac{k+1}{7}\rceil}7^{\lfloor\frac{7n-k-1}{4}\rfloor}\cdot \widetilde{b_k}(n,2)t^n.
\end{align}
\end{lemma}

For the sake of convenience in writing, we define $p_2=1$.
\begin{lemma}\label{p0+p1}
Based on the equations in the Appendix, we have
\begin{align}
U_{A_5}(p_0t^k)=&c_{k,0}(3,0)(p_0+p_1)t^3+\sum_{i=0}^{2}\sum_{n\geq4}c_{k,0}(n,i)7^{\lfloor\frac{7n-k+\omega_{0,i}}{4}\rfloor}p_it^n, \text{where}~k=1~or~2;\label{p0-p0+p1}\\
U_{A_5}(p_1t^k)=&c_{k,1}(3,0)(p_0+p_1)t^3+\sum_{i=0}^{2}\sum_{n\geq4}c_{k,1}(n,i)7^{\lfloor\frac{7n-k+\omega_{1,i}}{4}\rfloor}p_it^n, \text{where}~k=2~or~3;\label{p1-p0+p1}\\
U_{A_5}(t^k)=&c_{k,2}(3,0)(p_0+p_1)t^3+\sum_{i=0}^{2}\sum_{n\geq4}c_{k,2}(n,i)7^{\lfloor\frac{7n-k+\omega_{2,i}}{4}\rfloor}p_it^n, \text{where}~k=1~or~2.\label{1-p0+p1}
\end{align}
According to \eqref{A5-p0}-\eqref{A5-1}, we get $\omega_{2,2}=\omega_{0,2}=\omega_{0,1}=\omega_{1,1}=-23$, $\omega_{1,2}=\omega_{2,1}=-22$, $\omega_{1,0}=\omega_{2,0}=-19$, $\omega_{0,0}=-20$.
\end{lemma}

\begin{proof}
We proof \eqref{p0-p0+p1} as an example. By $U_{A_5}(ut^N)=-\sum _{i=0}^6a_i(t)U_{A_5}(ut^{N-7+i})$, we have $U_{A_5}(p_ut^k)=-\sum _{i=0}^6a_i(t)U_{A_5}(p_ut^{k-7+i})$. Since $a_j(t)$ is a polynomial of $t$, $ord_t(a_j(t))\geq1$; we just need to focus on which items can produce $p_0t^2$ and $p_1t^2$. By Lemma \ref{5-concrete-relations}, for $k=1$, just $U_{A_5}(p_0t^{-6})$ and $U_{A_5}(p_0t^{-5})$ satisfy the condition and the coefficients of $p_0t^2$ and $p_1t^2$ are equal. Thus $U_{A_5}(p_0t)$ can be written as \eqref{p0-p0+p1}.
\end{proof}

In the following lemma, we briefly describe the coefficients of the first three terms of $L_{\alpha}$.
\begin{lemma}\label{L1-L3}
Given $L_{\alpha}$ in definition \eqref{5-L-definition}, we have
\begin{align}
   \label{L1} L_1\equiv & 21t^4\pmod {7^2},\\
   \label{L2} L_2\equiv &14t \pmod {7^2},\\
   \label{L3} L_3\equiv &35(p_0+p_1)t^3 \pmod {7^2}.
\end{align}
\end{lemma}

\begin{proof}
We prove equation \eqref{L1} as an example. Since $L_1=U_{A_5}(1)$, we extract the items whose coefficients can not be divided by $7^2$ in \eqref{A5-1} when $k=0$. Then we have to focus on the coefficients of $t^4$, $p_0t^3$, $p_1t^3$, and $p_1t^4$, respectively. According to \eqref{re-t}, we get that $U_{A_5}(1)$ can be expressed by $U_{A_5}(t^{j})$, $-7\leq j\leq -1$. The parts of $t^4$, $p_0t^3$, and $p_1t^3$ are all related to $U_{A_5}(t^{-7})$, $U_{A_5}(t^{-6})$, and $U_{A_5}(t^{-5})$. We can easily prove that the coefficients of $p_0t^3$ and $p_1t^3$ can be divided by $7^2$ by the Appendix. With the help of $7^2\mid a_6(t)$ and $7^2 \mid a_7(t)$, the divisibility of $7^2$ about  the coefficients of $p_1t^4$ in $U_{A_5}(1)$ can be proved by considering $U_{A_5}(t^{j})$, $-7\leq j\leq -3$. According to the Group VI in the Appendix, we get
\begin{align*}
L_1\equiv(6\cdot7+46\cdot7)t^4\equiv 21t^4\pmod {7^2}
\end{align*}
which completes the proof of \eqref{L1}.
\end{proof}

\begin{lemma}\label{L-odd even}
For $\alpha\geq1$ there exist integers $d_{n,i}^{(\alpha)}$, $i=0,1,2$ such that
\begin{align}\label{L-odd-5}
L_{2\alpha-1}=&7d_{3,0}^{(2\alpha-1)}(p_0+p_1)t^3+p_0\sum_{n\geq4}d_{n,0}^{(2\alpha-1)}7^{\lfloor\frac{7n-20}{4}\rfloor}t^n\\ \nonumber&+p_1\sum_{n\geq4}d_{n,1}^{(2\alpha-1)}7^{\lfloor\frac{7n-24}{4}\rfloor}t^n+\sum_{n\geq4}d_{n,2}^{(2\alpha-1)}7^{\lfloor\frac{7n-24}{4}\rfloor}t^n,
\end{align}
and
\begin{align}\label{L-even-5}
L_{2\alpha}=p_0\sum_{n\geq1}d_{n,0}^{(2\alpha)}7^{\lfloor\frac{7n+2}{4}\rfloor}t^n+p_1\sum_{n\geq2}d_{n,1}^{(2\alpha)}7^{\lfloor\frac{7n-1}{4}\rfloor}t^n+\sum_{n\geq1}d_{n,2}^{(2\alpha)}7^{\lfloor\frac{7n-2}{4}\rfloor}t^n.
\end{align}
\end{lemma}

\begin{proof}
For $\alpha=1$, we have
\begin{align*}
U_{A_5}(1)=p_0\sum_{n\geq3}7^{\lfloor\frac{7n-19}{4}\rfloor}c_{0,2}(n,0)t^n+p_1\sum_{n\geq3}7^{\lfloor\frac{7n-22}{4}\rfloor}c_{0,2}(n,1)t^n+\sum_{n\geq4}7^{\lfloor\frac{7n-22}{4}\rfloor}c_{0,2}(n,2)t^n.
\end{align*}
Using specific data, we have $d_{3,0}^{(1)}=d_{3,1}^{(1)}=21$, then equation \eqref{L-odd-5} holds for $\alpha=1$.\par
Next, we assume that \eqref{L-odd-5} holds for all $\alpha\leq s$, then we consider $\alpha=s$, then we consider \eqref{L-even-5} when $\alpha=s$, then
\begin{align}\label{U-B-odd}
 L_{2s}=U_B(L_{2s-1})=&\sum_{k\geq4}d_{k,0}^{(2s-1)} 7^{\lfloor\frac{7k-20}{4}\rfloor}U_B(p_0t^k)+\sum_{k\geq4}d_{k,1}^{(2s-1)}7^{\lfloor\frac{7k-24}{4}\rfloor}U_B(p_1t^k)\\ &+\sum_{k\geq4}d_{k,2}^{(2s-1)}7^{\lfloor\frac{7k-24}{4}\rfloor}U_B(t^k)+7d_{3.0}^{(2s-1)}U_B((p_0+p_1)t^3).\nonumber
\end{align}
We aim to show that the form of each sum on the right-hand side of the equation \eqref{U-B-odd} satisfies the form in \eqref{L-even-5}. Since the proofs are similar, we just consider the first sum. From \eqref{B-p0t}, we have
\begin{align*}
 \sum_{k\geq4}d_{k,0}^{(2s-1)} 7^{\lfloor\frac{7n-20}{4}\rfloor}U_B(p_0t^k)=p_0\sum_{k\geq4}\sum_{n\geq\lceil\frac{k-3}{7}\rceil}b_{k,0}(n,0)d_{k,0}^{(2s-1)}7^{\lfloor\frac{7k-20}{4}\rfloor+\lfloor\frac{7n-k+2}{4}\rfloor}t^n\\+p_1\sum_{k\geq4}\sum_{n\geq\lceil\frac{k+4}{7}\rceil}b_{k,0}(n,1)d_{k,0}^{(2s-1)}7^{\lfloor\frac{7k-20}{4}\rfloor+\lfloor\frac{7n-k-1}{4}\rfloor}t^n\\+\sum_{k\geq4}\sum_{n\geq\lceil\frac{k+1}{7}\rceil}b_{k,0}(n,2)d_{k,0}^{(2s-1)}7^{\lfloor\frac{7k-20}{4}\rfloor+\lfloor\frac{7n-k-1}{4}\rfloor}t^n.
\end{align*}
For $k\geq4$, we have
\begin{align*}
 \lfloor\frac{7k-20}{4}\rfloor+\lfloor\frac{7n-k+2}{4}\rfloor\geq2+\lfloor\frac{7n-2}{4}\rfloor\geq\lfloor\frac{7n+2}{4}\rfloor,\\\lfloor\frac{7k-20}{4}\rfloor+\lfloor\frac{7n-k-1}{4}\rfloor\geq2+\lfloor\frac{7n-5}{4}\rfloor\geq\lfloor\frac{7n-1}{4}\rfloor.
\end{align*}
So that the first sum in the right side of \eqref{U-B-odd} has the form of \eqref{L-even-5}. Similarly, also the second and third sums have the correct form. From \eqref{B-p0+p1},
\begin{align*}
U_{B}((p_0+p_1)t^3)=p_0\sum_{n\geq1}\widetilde{b_3}(n,0)7^{\lfloor\frac{7n-1}{4}\rfloor}t^n+p_1\sum_{n\geq2}\widetilde{b_3}(n,1)7^{\lfloor\frac{7n-4}{4}\rfloor}t^n+\sum_{n\geq1}\widetilde{b_3}(n,2)7^{\lfloor\frac{7n-4}{4}\rfloor}t^n.
\end{align*}
For $k=3$, we have
\begin{align*}
1+\lfloor\frac{7n-1}{4}\rfloor\geq\lfloor\frac{7n+2}{4}\rfloor,\\
1+\lfloor\frac{7n-4}{4}\rfloor\geq\lfloor\frac{7n-1}{4}\rfloor.
\end{align*}
Thus the forth formula in the right side of \eqref{U-B-odd} has the form of \eqref{L-even-5}. Hence equation \eqref{L-even-5} holds for $\alpha=s$.\par
Last, we assume \eqref{L-even-5} holds for all $\alpha\leq s$, then we consider \eqref{L-odd-5} for $\alpha=s+1$, then
\begin{align}\label{U-A-even}
L_{2s+1}=U_{A_5}(L_{2s})=&\sum_{k\geq1}d_{k,0}^{(2s)}7^{\lfloor\frac{7k+2}{4}\rfloor}U_{A_5}(p_0t^k)+\sum_{k\geq2}d_{k,1}^{(2s)}7^{\lfloor\frac{7k-1}{4}\rfloor}U_{A_5}(p_1t^k)\\ \nonumber&+\sum_{k\geq1}d_{k,2}^{(2s)}7^{\lfloor\frac{7k-2}{4}\rfloor}U_{A_5}(t^k).
\end{align}
We aim to show that the form of each sum on the right-hand side of the equation \eqref{U-A-even} satisfies the form in \eqref{L-odd-5}. Since the proofs are similar, we just consider the first sum. From \eqref{A5-p0}, we get
\begin{align*}
\sum_{k\geq1}d_{k,0}^{(2s)}7^{\lfloor\frac{7k+2}{4}\rfloor}U_{A_5}(p_0t^k)=p_0\sum_{k\geq1}\sum_{n\geq\lceil\frac{k+19}{7}\rceil}a_{k,0}(n,0)d_{k,0}^{(2s)}7^{\lfloor\frac{7n-k-20}{4}\rfloor+\lfloor\frac{7k+2}{4}\rfloor}t^n\\+ p_1\sum_{k\geq1}\sum_{n\geq\lceil\frac{k+19}{7}\rceil}a_{k,0}(n,1)d_{k,0}^{(2s)}7^{\lfloor\frac{7n-k-23}{4}\rfloor+\lfloor\frac{7k+2}{4}\rfloor}t^n\\+\sum_{k\geq1}\sum_{n\geq\lceil\frac{k+26}{7}\rceil}a_{k,0}(n,2)d_{k,0}^{(2s)}7^{\lfloor\frac{7n-k-23}{4}\rfloor+\lfloor\frac{7k+2}{4}\rfloor}t^n.
\end{align*}
For $k\geq1$, we have
\begin{align*}
\lfloor\frac{7n-k-20}{4}\rfloor+\lfloor\frac{7k+2}{4}\rfloor\geq2+\lfloor\frac{7n-21}{4}\rfloor\geq\lfloor\frac{7n-20}{4}\rfloor,\\\lfloor\frac{7n-k-23}{4}\rfloor+\lfloor\frac{7k+2}{4}\rfloor\geq2+\lfloor\frac{7n-24}{4}\rfloor\geq\lfloor\frac{7n-24}{4}\rfloor.
\end{align*}
So that the first sum in the right side of \eqref{U-A-even} has the form of \eqref{L-odd-5}. Similarly, also the second and third sums have the correct form. For $U_{A_5}(p_it^k)$, just $U_{A_5}(p_0t)$, $U_{A_5}(p_0t^2)$, $U_{A_5}(t)$, $U_{A_5}(t^2)$, $U_{A_5}(p_1t^2)$, and $U_{A_5}(p_1t^3)$ can have items of $p_0t^3$ and $p_1t^3$. By lemma \ref{p0+p1}, they both can be written as $a(p_0+p_1)t^3$. Thus for $\alpha=s+1$, $p_0t^3$ and $p_1t^3$ have same coefficients.

We let
\begin{align*}
L_{2\alpha-1}=&p_0\sum_{n=3}^{\infty}l_{n,0}^{(2\alpha-1)}t^n+p_1\sum_{n=3}^{\infty}l_{n,1}^{(2\alpha-1)}t^n+\sum_{n=4}^{\infty}l_{n,2}^{(2\alpha-1)}t^n,\\
L_{2\alpha}=&p_0\sum_{n=1}^{\infty}l_{n,0}^{(2\alpha)}t^n+p_1\sum_{n=2}^{\infty}l_{n,1}^{(2\alpha)}t^n+\sum_{n=1}^{\infty}l_{n,2}^{(2\alpha)}t^n.
\end{align*}
Also, based on the Appendix we easily find that the equations above have integer coefficients and $l_{3,0}^{(2\alpha+1)}=l_{3,1}^{(2\alpha+1)}\equiv0(mod~7)$. The general result follows by induction.
\end{proof}
Next we define
\begin{align*}
D^{(\alpha)}(l_{m,i},l_{n,j}):=l_{m,i}^{(\alpha)}l_{n,j}^{(\alpha+2)}-l_{m,i}^{(\alpha+2)}l_{n,j}^{(\alpha)},
\end{align*}
and denote $\pi(n)$ be the 7-adic order of $n$ (i.e. the highest power of 7 that divides n).
\begin{lemma}\label{5-main lamma}
For $\alpha\geq1$, $i,j\in\{0,1,2\}$, the following inequalities hold:
\begin{align}\label{pi-odd}
\pi(D^{(2\alpha-1)}(l_{m,i},l_{n,j})\geq\alpha-1+m+n-6+\lambda_i+\lambda_j+\delta(m,i,n,j),
\end{align}
\begin{align}\label{pi-even}
\pi(D^{(2\alpha)}(l_{m,i},l_{n,j})\geq\alpha+m+n+\lambda_i\lambda_j,
\end{align}
where
\begin{align*}
\lambda_i=&
\begin{cases}
    1,\quad i=0,\\
    0,\quad i=1~or~2,\\
\end{cases}\\
\delta(m,i,n,j)=&
\begin{cases}
    0,\quad else,\\
    1,\quad (m,i)~or~(n,j)=(3,1).\\
\end{cases}
\end{align*}
\end{lemma}
\begin{proof}
We prove this lemma by mathematical induction. Firstly, we consider the case of $\alpha=1$. By the expansion of $L_{2\alpha-1}$, we have
\begin{align*}
\pi(D^{(1)}(l_{m,2},l_{n.2}))&\geq\lfloor\frac{7m-24}{4}\rfloor+\lfloor\frac{7n-24}{4}\rfloor\geq m+n-6,\\
\pi(D^{(1)}(l_{m,2},l_{n.1}))&\geq\lfloor\frac{7m-24}{4}\rfloor+\lfloor\frac{7n-24}{4}\rfloor\geq m+n-6,\\
\pi(D^{(1)}(l_{m,2},l_{n.0}))&\geq\lfloor\frac{7m-24}{4}\rfloor+\lfloor\frac{7n-20}{4}\rfloor\geq m+n-5,\\
\pi(D^{(1)}(l_{m,1},l_{n.1}))&\geq\lfloor\frac{7m-24}{4}\rfloor+\lfloor\frac{7n-24}{4}\rfloor\geq m+n-6,\\
\pi(D^{(1)}(l_{m,1},l_{n.0}))&\geq\lfloor\frac{7m-24}{4}\rfloor+\lfloor\frac{7n-20}{4}\rfloor\geq m+n-5,\\
\pi(D^{(1)}(l_{m,0},l_{n.0}))&\geq\lfloor\frac{7m-20}{4}\rfloor+\lfloor\frac{7n-20}{4}\rfloor\geq m+n-4
\end{align*}
when $m,n\geq4$,
\begin{align*}
\pi(D^{(1)}(l_{m,2},l_{3,1}))&\geq\lfloor\frac{7m-24}{4}\rfloor+1\geq m-2,\\
\pi(D^{(1)}(l_{m,2},l_{3,0}))&\geq\lfloor\frac{7m-24}{4}\rfloor+1\geq m-2,\\
\pi(D^{(1)}(l_{m,1},l_{3,1}))&\geq\lfloor\frac{7m-24}{4}\rfloor+1\geq m-2,\\
\pi(D^{(1)}(l_{m,1},l_{3,0}))&\geq\lfloor\frac{7m-24}{4}\rfloor+1\geq m-2,\\
\pi(D^{(1)}(l_{m,0},l_{3,1}))&\geq\lfloor\frac{7m-20}{4}\rfloor+1\geq m-1,\\
\pi(D^{(1)}(l_{m,0},l_{3,0}))&\geq\lfloor\frac{7m-20}{4}\rfloor+1\geq m-1
\end{align*}
when $m\geq4$, and
\begin{align*}
\pi(D^{(1)}(l_{3,1},l_{3,0}))&=\infty,
\end{align*}
which prove \eqref{pi-odd} for $\alpha=1$.

Secondly, we assume that \eqref{pi-odd} holds for all $\alpha\leq s$, then we consider \eqref{pi-even} when $\alpha=s$. We now compare the coefficients on both sides of
\begin{align*}
L_{2s}=&p_0\sum_{n=1}^{\infty}l_{n,0}^{(2s)}t^n+p_1\sum_{n=2}^{\infty}l_{n,1}^{(2s)}t^n+\sum_{n=1}^{\infty}l_{n,2}^{(2s)}t^n\\=&\sum_{n=3}^{\infty}l_{n,0}^{(2s-1)}U_B(p_0t^n)+\sum_{n=3}^{\infty}l_{n,1}^{(2s-1)}U_B(p_1t^n)+\sum_{n=4}^{\infty}l_{n,2}^{(2s-1)}U_B(t^n).
\end{align*}
For $u=0,1~or~2$, define $x_{k,u}(m,i):=[p_it^m]U_B(p_ut^k)$ be the coefficient of $p_it^m$ in $U_B(p_ut^k)$, namely
\begin{align*}
U_B(p_ut^k)=p_0\sum_nx_{k,u}(n,0)t^n+p_1\sum_nx_{k,u}(n,1)t^n+\sum_nx_{k,u}(n,2)t^n.
\end{align*}
Then we derive that
\begin{align*}
l_{n,i}^{(2s)}=\sum_{k,u}x_{k,u}(n,i)l_{k,u}^{(2s-1)}
\end{align*}
and
\begin{align*}
D^{(2s)}(l_{m,i},l_{n,j})=\sum_{k,u,r,v}x_{k,u}(m,i)x_{r,v}(n,j)D^{(2s-1)}(l_{k,u},l_{r,v}).
\end{align*}
By \eqref{B-p0t}-\eqref{B-t}, we get $\pi(x_{k,u}(m,i))\geq\lfloor\frac{7m-k+\gamma_{u,i}}{4}\rfloor$, where $\gamma_{2,2}=\gamma_{0,2}=\gamma_{0,1}=\gamma_{1,1}=-1$, $\gamma_{1,2}=0$ and $\gamma_{0,0}=\gamma_{1,2}=2$. We obtain that
\begin{align*}
&\pi(D^{(2s)}(l_{m,i},l_{n,i})\\
=&\pi(\sum_{k,u,r,v}x_{k,u}(m,i)x_{r,v}(n,j)D^{(2s-1)}(l_{k,u},l_{r,v}))\\
\geq&\min_{(k,u)\ne(r,v)}(s-1+m+n-6+\lambda_u+\lambda_v+\delta(k,u,r,v)+\lfloor\frac{3m+3k+\gamma_{u,i}}{4}\rfloor+\lfloor\frac{3n+3r+\gamma_{v,j}}{4}\rfloor)\\
\geq&s+m+n+\lambda_i\lambda_j.
\end{align*}
Thus we get that inequality \eqref{pi-even} holds for $\alpha=s$.\par
Finally, we assume that \eqref{pi-even} holds for all $\alpha\leq s$, then we consider \eqref{pi-odd} when $\alpha=s+1$. For $u=0,1~or~2$, define $y_{k,u}(m,i):=[p_it^m]U_A(p_ut^k)$ be the coefficient of $p_it^m$ in $U_A(p_ut^k)$, namely
\begin{align*}
U_A(p_ut^k)=p_0\sum_ny_{k,u}(n,0)t^n+p_1\sum_ny_{k,u}(n,1)t^n+\sum_ny_{k,u}(n,2)t^n.
\end{align*}
Then we derive that
\begin{align*}
l_{n,i}^{(2s+1)}=\sum_{k,u}y_{k,u}(n,i)l_{k,u}^{(2s)}
\end{align*}
and
\begin{align*}
D^{(2s+1)}(l_{m,i},l_{n,j})=\sum_{k,u,r,v}y_{k,u}(m,i)y_{r,v}(n,j)D^{(2s)}(l_{k,u},l_{r,v}).
\end{align*}
By \eqref{A5-p0}--\eqref{A5-1}, we get $\pi(y_{k,u}(m,i))\geq\lfloor\frac{7m-k+\omega_{u,i}}{4}\rfloor$, where $\omega_{2,2}=\omega_{0,2}=\omega_{0,1}=\omega_{1,1}=-23$, $\omega_{1,2}=\omega_{2,1}=-22$, $\omega_{1,0}=\omega_{2,0}=-19$, $\omega_{0,0}=-20$. We obtain that
\begin{align}\label{D-odd}
&\pi(D^{(2s+1)}(l_{m,i},l_{n,i})\\
\nonumber=&\pi(\sum_{k,u,r,v}y_{k,u}(m,i)y_{r,v}(n,j)D^{(2s)}(l_{k,u},l_{r,v}))\\
\nonumber\geq&\min_{(k,u)\ne(r,v)}(s+m+n+\lfloor\frac{3m+3k+\omega_{u,i}}{4}\rfloor+\lfloor\frac{3n+3r+\omega_{v,j}}{4}\rfloor)\\
\nonumber\geq&s+m+n-6+\lambda_i+\lambda_j+\delta(m,i,n,j).
\end{align}
By \eqref{D-odd}, we get that inequality \eqref{pi-odd} holds for $\alpha=s+1$.
\end{proof}

\begin{theorem}\label{congruence equation}
For each $\alpha\geq1$ there exist an integral constant $x_{\alpha}$ such that
\begin{align}
\label{L-even-A5}L_{2\alpha+2}&\equiv x_{2\alpha}L_{2\alpha}(mod~7^{\alpha+1}),\\
\label{L-odd-A5} L_{2\alpha+1}&\equiv x_{2\alpha-1}L_{2\alpha-1}(mod~7^{\alpha}).
\end{align}
\end{theorem}
\begin{proof}
By Lemma \ref{5-main lamma}, we have
\begin{align*}
\pi(D^{(2\alpha-1)}(l_{m,i},l_{n,j}))\geq\alpha+1,\\
\pi(D^{(2\alpha)}(l_{m,i},l_{n,j}))\geq\alpha+2.
\end{align*}
According to the Appendix, we have
\begin{align}\label{UAB}
U_B(U_{A_5}(t))\equiv 4t(mod~7),~U_{A_5}(U_B((p_0+p_1)t^3))\equiv 4(p_0+p_1)t^3(mod~7).
\end{align}
By Lemma \ref{L1-L3}, we have $L_1\equiv3\cdot7t^4(mod~7^2)$,  $L_2\equiv2\cdot7t(mod~7^2)$, and $L_3\equiv5\cdot7(p_0+p_1)t^3(mod~7^2)$, which indicates that \eqref{L-odd-A5} holds for $\alpha=1$.

From \eqref{UAB}, for $\alpha\geq1$, we have
\begin{align*}
&L_{2\alpha}\equiv l_{1,2}^{(2\alpha)}t\equiv2^{2\alpha-1}\cdot7t(mod~7^2),\\ &L_{2\alpha+1}\equiv l_{3,0}^{(2\alpha+1)}(p_0+p_1)t^3\equiv2^{2\alpha-2}\cdot5\cdot7(p_0+p_1)t^3(mod~7^2),
\end{align*}
This implies that $7^2\nmid l_{1,2}^{(2\alpha)}$ and $7^2\nmid l_{3,0}^{(2\alpha+1)}$. Let $x_{2\alpha}$ be a solution of $l_{1,2}^{(2\alpha)}\equiv x_{2\alpha}l_{1,2}^{(2\alpha+2)}(mod~7^{\alpha+1})$, then for $(n,i)\ne(1,2)$, $7^2\mid l_{n,i}^{(2\alpha)}$ and
\begin{align*}
l_{n,i}^{(2\alpha)}l_{1,2}^{(2\alpha)}\equiv l_{n,i}^{(2\alpha)}x_{2\alpha}l_{1,2}^{(2\alpha+2)}\equiv x_{2\alpha}l_{1,2}^{(2\alpha)}l_{n,i}^{(2\alpha+2)}(mod~7^{\alpha+2}).
\end{align*}
Cancelling $l_{1,2}^{(2\alpha)}$, we obtain $l_{n,i}^{(2\alpha)}\equiv x_{2\alpha}l_{n,i}^{(2\alpha+2)}(mod~7^{\alpha+1})$. Similarly, let $x_{2\alpha+1}$ be a solution of $l_{3,0}^{(2\alpha+1)}\equiv x_{2\alpha+1}l_{3,0}^{(2\alpha+3)}(mod~7^{\alpha+1})$, then for $(n,i)\ne(3,0)~and~(3,1)$, $7^2\mid l_{n,i}^{(2\alpha)}$ and
\begin{align*}
l_{n,i}^{(2\alpha+1)}l_{3,0}^{(2\alpha+1)}\equiv l_{n,i}^{(2\alpha+1)}x_{2\alpha+1}l_{3,0}^{(2\alpha+3)}\equiv x_{2\alpha+1}l_{3,0}^{(2\alpha+1)}l_{n,i}^{(2\alpha+3)}(mod~7^{\alpha+2}).
\end{align*}
Cancelling $l_{3,0}^{(2\alpha+1)}$, we obtain $l_{n,i}^{(2\alpha+1)}\equiv x_{2\alpha+1}l_{n,i}^{(2\alpha+3)}(mod~7^{\alpha+1})$. Thus we complete the proof of Theorem \ref{congruence equation}.
\end{proof}
Based on \eqref{D5-L-alpha} and Theorem \ref{congruence equation}, we can directly obtain the Theorem \ref{D5-7-power}.

\subsection*{Acknowledgements}
The first author was  supported by the National Key R\&D Program of China (Grant No. 2024YFA1014500) and the National Natural Science Foundation of China (Grant No. 12201387).

\section*{Appendix. The fundamental relations}

Group I:
\begin{align*}
U_{A_3}(p_0t^{-8})=&p_0(-6t+2490\cdot7^2t^2-499\cdot7^4t^3-265\cdot7^6t^4-29\cdot7^7t^5+7^{10}t^6)+p_1(t-4162\cdot7t^2\\
&-3188\cdot7^3t^3-662\cdot7^5t^4-36\cdot7^7t^5-8\cdot7^8t^6+7^{10}t^7)+(764\cdot7t^2-1285\cdot7^3t^3\\
&-540\cdot7^5t^4-123\cdot7^7t^5-184\cdot7^8t^6-2\cdot7^{11}t^7);\\
U_{A_3}(p_0t^{-7})=&p_0(t-324\cdot7^2t^2+61\cdot7^4t^3+4\cdot7^7t^4-7^7t^5-7^{9}t^6)+p_1(593\cdot7t^2+274\cdot7^3t^3\\
&+51\cdot7^5t^4+2\cdot7^7t^5+7^8t^6)+(-158\cdot7t^2+6\cdot7^4t^3-30\cdot7^5t^4-6\cdot7^7t^5+7^9t^6\\
&+7^{10}t^7);\\
U_{A_3}(p_0t^{-6})=&p_0(33\cdot7^2t^2-62\cdot7^3t^3-24\cdot7^5t^4-7^7t^5)+p_1(-78\cdot7t^2-10\cdot7^3t^3-12\cdot7^4t^4)\\
&+(29\cdot7t^2+10\cdot7^3t^3+96\cdot7^4t^4+19\cdot7^6t^5+7^8t^6);\\
U_{A_3}(p_0t^{-5})=&p_0(-11\cdot7t^2+6\cdot7^3t^3+7^5t^4)+p_1(10\cdot7t^2-16\cdot7^2t^3-2\cdot7^4t^4)+(-5\cdot7t^2\\
&-16\cdot7^2t^3-11\cdot7^4t^4-7^6t^5);\\
U_{A_3}(p_0t^{-4})=&p_0(-4\cdot7t^2-7^3t^3)+p_1(-12t^2+2\cdot7^2t^3)+(6t^2+3\cdot7^2t^3+7^4t^4);\\
U_{A_3}(p_0t^{-3})=&8\cdot7p_0t^2+12p_1t^2-7^2t^3;\\
U_{A_3}(p_0t^{-2})=&p_0(13\cdot7t^2+2584\cdot7^2t^3+5702\cdot7^4t^4+4255\cdot7^6t^5+214\cdot7^{9}t^6+276\cdot7^{10}t^7\\
&+26\cdot7^{12}t^8+7^{14}t^9)+p_1(10t^2+76\cdot7^3t^3+9704\cdot7^3t^4+1122\cdot7^6t^5+2892\cdot7^7t^6\\
&+78\cdot7^{10}t^7+52\cdot7^{11}t^8+2\cdot7^{13}t^9)+(96\cdot7^2t^3+2784\cdot7^3t^4+2707\cdot7^5t^5+160\cdot7^8t^6\\
&+232\cdot7^9t^7+24\cdot7^{11}t^8+7^{13}t^9).
\end{align*}

Group II:
\begin{align*}
U_{A_3}(p_1t^{-8})=&p_0(-2\cdot7t-7109\cdot7^2t^2-598\cdot7^4t^3-254\cdot7^6t^4-62\cdot7^8t^5+122\cdot7^{9}t^6+25\cdot7^{11}t^7\\
&+7^{13}t^8)+p_1(-18t+9958\cdot7t^2+5512\cdot7^3t^3+968\cdot7^5t^4+96\cdot7^7t^5+18\cdot7^8t^6\\
&-18\cdot7^{10}t^7-7^{12}t^8)+(-143\cdot7^2t^2+4666\cdot7^3t^3+2292\cdot7^5t^4+772\cdot7^7t^5+80\cdot7^9t^6\\
&-122\cdot7^{10}t^7-25\cdot7^{12}t^8-7^{14}t^9);\\
U_{A_3}(p_1t^{-7})=&p_0(108\cdot7^3t^2-19\cdot7^5t^3+3\cdot7^7t^4+23\cdot7^8t^5+2\cdot7^{10}t^6)+p_1(t-4\cdot7^4t^2-66\cdot7^4t^3\\
&-8\cdot7^6t^4-10\cdot7^7t^5-2\cdot7^9t^6)+(32\cdot7^2t^2-19\cdot7^4t^3+8\cdot7^6t^4-9\cdot7^8t^5-24\cdot7^9t^6\\
&-2\cdot7^{11}t^7);\\
U_{A_3}(p_1t^{-6})=&p_0(-50\cdot7^2t^2+62\cdot7^4t^3+10\cdot7^6t^4-7^8t^5-7^9t^6)+p_1(176\cdot7t^2+16\cdot7^3t^3-4\cdot7^5t^4\\
&+7^8t^6)+(-6\cdot7^2t^2-62\cdot7^3t^3-10\cdot7^6t^4-8\cdot7^7t^5+7^9t^6+7^{10}t^7);\\
U_{A_3}(p_1t^{-5})=&p_0(-3\cdot7^2t^2-94\cdot7^3t^3-19\cdot7^5t^4-7^7t^5)+p_1(-22\cdot7t^2+24\cdot7^2t^3+12\cdot7^4t^4+7^6t^5)\\
&+(7^2t^2+18\cdot7^3t^3+108\cdot7^4t^4+19\cdot7^6t^5+7^8t^6);\\
U_{A_3}(p_1t^{-4})=&p_0(13\cdot7t^2+11\cdot7^3t^3+7^5t^4)+p_1(22t^2-4\cdot7^2t^3-7^4t^4)+(-7t^2-20\cdot7^2t^3\\
&-11\cdot7^4t^4-7^6t^5);\\
U_{A_3}(p_1t^{-3})=&p_0(-3\cdot7t^2-7^3t^3)+p_1(-4t^2+7^2t^3)+(3\cdot7^2t^3+7^4t^4);\\
U_{A_3}(p_1t^{-2})=&p_0(-6\cdot7t^2-27\cdot7^3t^3-10\cdot7^5t^4-7^7t^5) +p_1(20t^2+2\cdot7^4t^3+65\cdot7^4t^4+2\cdot7^7t^5\\
&+7^8t^6)+(-29\cdot7^2t^3-16\cdot7^4t^4-2\cdot7^6t^5).
\end{align*}

Group III:
\begin{align*}
U_{A_3}(t^{-7})=&p_0(3t+148\cdot7^3t^2+215\cdot7^4t^3-24\cdot7^6t^4-31\cdot7^7t^5-2\cdot7^{9}t^6)+p_1(-1333\cdot7t^2\\
&-842\cdot7^3t^3-123\cdot7^5t^4-2\cdot7^7t^5-7^8t^6)+(18\cdot7^2t^2-814\cdot7^3t^3-398\cdot7^5t^4-6\cdot7^8t^5\\
&-2\cdot7^8t^6+7^{10}t^7);\\
U_{A_3}(t^{-6})=&p_0(-159\cdot7^2t^2-34\cdot7^4t^3)+p_1(30\cdot7^2t^2+94\cdot7^3t^3+12\cdot7^5t^4)+(-27\cdot7t^2+106\cdot7^3t^3\\
&+40\cdot7^5t^4+4\cdot7^7t^5+7^8t^6);\\
U_{A_3}(t^{-5})=&p_0(3\cdot7^3t^2+26\cdot7^3t^3)+p_1(-30\cdot7t^2-8\cdot7^3t^3-6\cdot7^4t^4)+(5\cdot7t^2-88\cdot7^2t^3\\
&-20\cdot7^4t^4-7^6t^5);\\
U_{A_3}(t^{-4})=&p_0(-16\cdot7t^2-2\cdot7^3t^3)+p_1(4\cdot7t^2+2\cdot7^2t^3)+(-6t^2+10\cdot7^2t^3+7^4t^4);\\
U_{A_3}(t^{-3})=&7p_0t^2-4p_1t^2+t^2-7^2t^3;\\
U_{A_3}(t^{-2})=&7p_0t^2+2p_1t^2;\\
U_{A_3}(t^{-1})=&p_0(2\cdot7t^2+373\cdot7^2t^3+816\cdot7^4t^4+608\cdot7^6t^5+214\cdot7^8t^6+276\cdot7^9t^7+26\cdot7^{11}t^8\\
&+7^{13}t^9)+p_1(t^2+74\cdot7^2t^3+1377\cdot7^3t^4+160\cdot7^6t^5+59\cdot7^8t^6+78\cdot7^9t^7+52\cdot7^{10}t^8\\
&+2\cdot7^{12}t^9)+(100\cdot7t^3+400\cdot7^3t^4+387\cdot7^5t^5+160\cdot7^7t^6+232\cdot7^8t^7+24\cdot7^{10}t^8\\
&+7^{12}t^9).
\end{align*}

Group IV:
\begin{align*}
U_{A_5}(p_0t^{-10})=&p_0(839\cdot7t^2+38296\cdot7^2t^3+4261\cdot7^4t^4-520\cdot7^6t^5-34\cdot7^{9}t^6-297\cdot7^9t^7\\
&-31\cdot7^{11}t^8-7^{13}t^9)+p_1(5872t^2+5597\cdot7^3t^3+105327\cdot7^3t^4+3204\cdot7^6t^5\\
&+451\cdot7^8t^6+264\cdot7^9t^7+61\cdot7^{10}t^8+7^{12}t^9)+(9\cdot7^3t^3+1663\cdot7^3t^4+669\cdot7^6t^5\\
&+2640\cdot7^7t^6+670\cdot7^{9}t^7+608\cdot7^{10}t^8+39\cdot7^{12}t^9+7^{14}t^{10});\\
U_{A_5}(p_0t^{-9})=&p_0(-1014t^2-4810\cdot7^2t^3+26\cdot7^4t^4+264\cdot7^6t^5+314\cdot7^7t^6+22\cdot7^9t^7+7^{11}t^8)\\
&+p_1(-1014t^2-31998\cdot7t^3-8356\cdot7^3t^4-1006\cdot7^5t^5-128\cdot7^7t^6-86\cdot7^8t^7\\
&-2\cdot7^{10}t^8)+(-66\cdot7t^3-144\cdot7^3t^4-78\cdot7^6t^5-264\cdot7^7t^6-360\cdot7^8t^7-31\cdot7^{10}t^8\\
&-7^{12}t^9);\\
U_{A_5}(p_0t^{-8})=&p_0(144t^2+583\cdot7^2t^3-13\cdot7^4t^4-5\cdot7^7t^5-32\cdot7^7t^6-7^9t^7)+p_1(144t^2+3448\cdot7t^3\\
&+614\cdot7^3t^4-18\cdot7^5t^5-4\cdot7^7t^6)+(9\cdot7t^3-5\cdot7^3t^4+7^7t^5+22\cdot7^7t^6+23\cdot7^8t^7\\
&+7^{10}t^8);\\
U_{A_5}(p_0t^{-7})=&p_0(-15t^2-67\cdot7^2t^3-55\cdot7^3t^4+8\cdot7^5t^5+7^7t^6)+p_1(-15t^2-323\cdot7t^3-69\cdot7^3t^4\\
&+4\cdot7^4t^5)+(-7t^3+7^4t^4-7^5t^5-7^7t^6-7^8t^7);\\
U_{A_5}(p_0t^{-6})=&p_0(38\cdot7t^3+2\cdot7^4t^4+7^5t^5)+p_1(10\cdot7t^3+38\cdot7^2t^4-2\cdot7^4t^5)+(-18\cdot7^2t^4-7^5t^5\\
&-7^6t^6);\\
U_{A_5}(p_0t^{-5})=&p_0(t^2+5\cdot7t^3)+p_1(t^2+67t^3+7^2t^4)+(t^3+4\cdot7^2t^4);\\
U_{A_5}(p_0t^{-4})=&p_0(76\cdot7t^3+7^5t^4+7^6t^5)+p_1(506t^3+36\cdot7^3t^4+4\cdot7^5t^5)+(22\cdot7^2t^4+7^5t^5).
\end{align*}

Group V:
\begin{align*}
U_{A_5}(p_1t^{-9})=&p_0(528\cdot7t^2+8535\cdot7^2t^3-2649\cdot7^4t^4-327\cdot7^7t^5-90\cdot7^9t^6-590\cdot7^{9}t^7-39\cdot7^{11}t^8\\
&-7^{13}t^9)+p_1(528\cdot7t^2+53540\cdot7t^3+13334\cdot7^3t^4+1958\cdot7^5t^5+372\cdot7^7t^6\\
&+424\cdot7^8t^7+32\cdot7^{10}t^8+7^{12}t^9)+(9\cdot7^3t^3+1663\cdot7^3t^4+669\cdot7^6t^5+2640\cdot7^7t^6\\
&+670\cdot7^9t^7+608\cdot7^{10}t^8+39\cdot7^{12}t^9+7^{14}t^{10});\\
U_{A_5}(p_1t^{-8})=&p_0(-558t^2-1014\cdot7^2t^3+344\cdot7^4t^4+254\cdot7^6t^5+362\cdot7^7t^6+31\cdot7^9t^7+7^{11}t^8)\\
&+p_1(-558t^2-5382\cdot7t^3-1016\cdot7^3t^4-48\cdot7^5t^5-18\cdot7^7t^6-22\cdot7^7t^7-7^{10}t^8)\\
&+(-66\cdot7t^3-144\cdot7^3t^4-78\cdot7^6t^5-264\cdot7^7t^6-360\cdot7^8t^7-31\cdot7^{10}t^8-7^{12}t^9);\\
U_{A_5}(p_1t^{-7})=&p_0(72t^2+111\cdot7^2t^3-23\cdot7^4t^4-3\cdot7^7t^5-23\cdot7^7t^6-7^9t^7)+p_1(72t^2+416\cdot7t^3\\
&+102\cdot7^3t^4+4\cdot7^5t^5+2\cdot7^7t^6+7^8t^7)+(9\cdot7t^3-5\cdot7^3t^4+7^7t^5+22\cdot7^7t^6\\
&+23\cdot7^8t^7+7^{10}t^8);\\
U_{A_5}(p_1t^{-6})=&p_0(-7t^2-9\cdot7^2t^3-9\cdot7^3t^4+7^6t^5+7^7t^6)+p_1(-7t^2+5\cdot7t^3-53\cdot7^2t^4+2\cdot7^4t^5\\
&-7^6t^6)+(-7t^3+7^4t^4-7^5t^5-7^7t^6-7^8t^7);\\
U_{A_5}(p_1t^{-5})=&p_0(7^4t^4+7^5t^5)+p_1(-90t^3-7^4t^5)+(-18\cdot7^2t^4-7^5t^5-7^6t^6);\\
U_{A_5}(p_1t^{-4})=&p_0(t^2-4\cdot7t^3)+p_1(t^2+3t^3)+(t^3+4\cdot7^2t^4);\\
U_{A_5}(p_1t^{-3})=&p_0(335\cdot7t^3+911\cdot7^3t^4+610\cdot7^5t^5+24\cdot7^8t^6+3\cdot7^{10}t^7+7^{11}t^8)\\
&+p_1(2346t^3+1009\cdot7^3t^4+5618\cdot7^4t^5+277\cdot7^7t^6+48\cdot7^9t^7+29\cdot7^{10}t^8+7^{12}t^9)\\
&+(22\cdot7^2t^4+7^5t^5).
\end{align*}

Group VI:
\begin{align*}
U_{A_5}(t^{-9})=&p_0(110\cdot7t^2+10025\cdot7^2t^3+2565\cdot7^4t^4+839\cdot7^6t^5+132\cdot7^{8}t^6+68\cdot7^9t^7+2\cdot7^{11}t^8)\\
&+p_1(110\cdot7t^2+10630\cdot7^2t^3+30440\cdot7^3t^4+5878\cdot7^5t^5+472\cdot7^7t^6+2\cdot7^{10}t^7\\
&-2\cdot7^{10}t^8)+(165\cdot7t^3+87\cdot7^3t^4-85\cdot7^6t^5-418\cdot7^{7}t^6-570\cdot7^8t^7-38\cdot7^{10}t^8-7^{12}t^9);\\
U_{A_5}(t^{-8})=&p_0(-240t^2-1291\cdot7^2t^3-129\cdot7^4t^4-37\cdot7^6t^5-40\cdot7^7t^6-2\cdot7^9t^7)+p_1(-240t^2\\
&-8872\cdot7t^3-2690\cdot7^3t^4-410\cdot7^5t^5-4\cdot7^8t^6-8\cdot7^8t^7)+(-33\cdot7t^3-33\cdot7^3t^4\\
&+7^6t^5+22\cdot7^7t^6+30\cdot7^8t^7+7^{10}t^8);\\
U_{A_5}(t^{-7})=&p_0(7^2t^2+160\cdot7^2t^3-7^6t^5)+p_1(7^2t^2+139\cdot7^2t^3+207\cdot7^3t^4+16\cdot7^5t^5)+(6\cdot7t^3\\
&+4\cdot7^3t^4+6\cdot7^5t^5-7^8t^7);\\
U_{A_5}(t^{-6})=&p_0(-8t^2-19\cdot7^2t^3-7^3t^4+7^5t^5)+p_1(-8t^2-94\cdot7t^3-18\cdot7^3t^4-6\cdot7^4t^5)+(-7t^3\\
&+3\cdot7^2t^4-7^5t^5-7^6t^6);\\
U_{A_5}(t^{-5})=&p_0(t^2+13\cdot7t^3+7^3t^4)+p_1(t^2+5\cdot7t^3+9\cdot7^2t^4)+(t^3-3\cdot7^2t^4);\\
U_{A_5}(t^{-4})=&10p_1t^3+7^2t^4;\\
U_{A_5}(t^{-3})=&p_0(11\cdot7t^3+7^4t^4+7^5t^5)+p_1(72t^3+36\cdot7^2t^4+4\cdot7^4t^5)+(3\cdot7^2t^4+7^{4}t^5).
\end{align*}

Group VII:
\begin{align*}
U_{B}(p_0t^{-6})=&p_0(19\cdot7t^{-1}+741\cdot7^2+513\cdot7^4t+76\cdot7^6t^2+7^{10}t^5)+p_1(-19\cdot7-114\cdot7^3t\\
&-114\cdot7^5t^2-19\cdot7^7t^3)+(-55\cdot7^2+475\cdot7^3t+95\cdot7^5t^2-7^{10}t^5-7^{11}t^6);\\
U_{B}(p_0t^{-5})=&p_0(-2\cdot7t^{-1}-78\cdot7^2-54\cdot7^4t-8\cdot7^6t^2+7^{8}t^4)+p_1(2\cdot7+12\cdot7^3t+12\cdot7^5t^2\\
&+2\cdot7^7t^3)+(186\cdot7-50\cdot7^3t-10\cdot7^5t^2-7^8t^4-7^9t^5);\\
U_{B}(p_0t^{-4})=&p_0(t^{-1}+39\cdot7+27\cdot7^3t+4\cdot7^5t^2+7^6t^3)+p_1(-1-6\cdot7^2t-6\cdot7^4t^2-7^6t^3)\\
&+(-319+25\cdot7^2t+5\cdot7^4t^2-7^6t^3-7^7t^4);\\
U_{B}(p_0t^{-3})=&7^4p_0t^2+64-7^4t^2-7^5t^3;\\
U_{B}(p_0t^{-2})=&7^2p_0t-12-7^2t-7^3t^2;\\
U_{B}(p_0t^{-1})=&p_0+3-7t;\\
U_{B}(p_0)=&p_0(78+496\cdot7^2t+74\cdot7^5t^2+1471\cdot7^5t^3+286\cdot7^7t^4+27\cdot7^9t^5+7^{11}t^6)+p_1(-11\cdot7t\\
&-64\cdot7^3t^2-58\cdot7^5t^3-134\cdot7^6t^4-19\cdot7^8t^5-7^{10}t^6)+(248\cdot7t+400\cdot7^3t^2\\
&+195\cdot7^5t^3+289\cdot7^6t^4+4\cdot7^9t^5+7^{10}t^6).
\end{align*}

Group VIII:
\begin{align*}
U_{B}(p_1t^{-6})=&p_0(-125t^{-1}-780\cdot7^2-982\cdot7^4t-590\cdot7^6t^2-1471\cdot7^7t^3-286\cdot7^9t^4-26\cdot7^{11}t^5\\
&-6\cdot7^{12}t^6)+p_1(18\cdot7+17\cdot7^4t+172\cdot7^5t^2+76\cdot7^7t^3+134\cdot7^8t^4+19\cdot7^{10}t^5\\
&+6\cdot7^{11}t^6)+(2325\cdot7-124\cdot7^3t+106\cdot7^6t^2+650\cdot7^7t^3+1615\cdot7^8t^4+6\cdot7^{12}t^5\\
&+26\cdot7^{12}t^6+6\cdot7^{13}t^7);\\
U_{B}(p_1t^{-5})=&p_0(2\cdot7t^{-1}+78\cdot7^2+54\cdot7^4t+8\cdot7^6t^2+7^9t^4+7^{10}t^5)+p_1(-2\cdot7-12\cdot7^3t\\
&-12\cdot7^5t^2-2\cdot7^7t^3-7^9t^5)+(-606\cdot7+50\cdot7^3t+10\cdot7^5t^2-7^{10}t^5-7^{11}t^6);\\
U_{B}(p_1t^{-4})=&p_0(-t^{-1}-39\cdot7-27\cdot7^3t-4\cdot7^5t^3+7^7t^3+7^8t^4)+p_1(1+6\cdot7^2t+6\cdot7^4t^2+7^6t^3\\
&-7^7t^4)+(121\cdot7-25\cdot7^2t-5\cdot7^4t^2-7^8t^4-7^9t^5);\\
U_{B}(p_1t^{-3})=&p_0(7^5t^2+7^6t^3)-7^5p_1t^3+(-148-7^6t^3-7^7t^4);\\
U_{B}(p_1t^{-2})=&p_0(7^3t+7^4t^2)-7^3p_1t^2+(24-7^4t^2-7^5t^3);\\
U_{B}(p_1t^{-1})=&p_0(7+7^2t)-7p_1t+(-4-7^2t-7^3t^2);\\
U_{B}(p_1)=&p_0(-38-27\cdot7^2t-4\cdot7^4t^2)+p_1(6\cdot7t+6\cdot7^3t^2+7^5t^3)+(-26\cdot7t-5\cdot7^3t^2).
\end{align*}

Group IX:
\begin{align*}
U_{B}(t^{-6})=&-2392\cdot7-7^{11}t^6;\\
U_{B}(t^{-5})=&68\cdot7-7^9t^5;\\
U_{B}(t^{-4})=&260-7^7t^4;\\
U_{B}(t^{-3})=&-88-7^5t^3;\\
U_{B}(t^{-2})=&20-7^3t^2;\\
U_{B}(t^{-1})=&-4-7t;\\
U_{B}(1)=&1.
\end{align*}







\begin{thebibliography}{10}
\bibitem{Andrews-1976}
G.~E. Andrews, {\it The theory of partitions}, Encyclopedia of Mathematics and its Applications, Vol. 2, Addison-Wesley Publishing Co., Reading, Mass.-London-Amsterdam, 1976; \MR{0557013}
\bibitem{Andrews-Paule-2022}
G.~E. Andrews and P. Paule, MacMahon's partition analysis XIII: Schmidt type partitions and modular forms, J. Number Theory {\bf 234} (2022), 95--119; \MR{4370530}
\bibitem{Atkin-1967}
A.~O.~L. Atkin, Proof of a conjecture of Ramanujan, Glasgow Math. J. {\bf 8} (1967), 14--32; \MR{0205958}
\bibitem{Atkin-Lehner-1970}
A.~O.~L. Atkin and J. Lehner, Hecke operators on $\Gamma \sb{0}(m)$, Math. Ann. {\bf 185} (1970), 134--160; \MR{0268123}
\bibitem{Banerjee-Smoot-2025}
K. Banerjee and N.~A. Smoot, The localization method applied to $k$-elongated plane partitions and divisibility by 5, Math. Z. {\bf 309} (2025), no.~3, Paper No. 46, 51 pp.; \MR{4852261}
\bibitem{Banerjee-Smoot-2025-7}
K. Banerjee and N.~A. Smoot, 2-Elongated Plane Partitions and Powers of 7: The Localization Method Applied to a Genus 1 Congruence Family, (2025), \url{https://arxiv.org/abs/2306.15594}
\bibitem{Baruah-Das_Talukdar-2023}
N.~D. Baruah, H. Das and P. Talukdar, Congruences for $k$-elongated plane partition diamonds, Int. J. Number Theory {\bf 19} (2023), no.~9, 2121--2139; \MR{4641757}
\bibitem{Chen-Chen-Garvan-2024}
D. Chen, R. Chen and F. Garvan, Congruences modulo powers of $5$ and $7$ for the crank and rank parity functions and related mock theta functions, (2024),
 \url{https://arxiv.org/abs/2407.07107}
\bibitem{Silva-Hirschhorn-Sellers-2022}
R. da~Silva, M.~D. Hirschhorn and J.~A. Sellers, Elementary proofs of infinitely many congruences for $k$-elongated partition diamonds, Discrete Math. {\bf 345} (2022), no.~11, Paper No. 113021, 12 pp.; \MR{4438132}
\bibitem{Garvan-1984}
F.~G. Garvan, A simple proof of Watson's partition congruences for powers of $7$, J. Austral. Math. Soc. Ser. A {\bf 36} (1984), no.~3, 316--334; \MR{0733905}
\bibitem{gtutorial}
F.~G. Garvan, A tutorial for the MAPLE ETA package, arXiv:1907.09130;
\bibitem{Hirschhorn-1981}
M.~D. Hirschhorn and D.~C. Hunt, A simple proof of the Ramanujan conjecture for powers of $5$, J. Reine Angew. Math. {\bf 326} (1981), 1--17; \MR{0622342}
\bibitem{Radu-2015}
C.-S. Radu, An algorithmic approach to Ramanujan-Kolberg identities, J. Symbolic Comput. {\bf 68} (2015), 225--253; \MR{3283845}
\bibitem{Ramanujan-1919}
S. Ramanujan, Some properties of $p(n)$, the number of partitions of $n$, Proc. Camb. Philos. Soc. 19, 214-216 (1919). \MR{2280868}
\bibitem{Smoot-2021}
N.~A. Smoot, On the computation of identities relating partition numbers in arithmetic progressions with eta quotients: an implementation of Radu's algorithm, J. Symbolic Comput. {\bf 104} (2021), 276--311; \MR{4180128}
\bibitem{Watson-1938}
G.~N. Watson, Ramanujans Vermutung \"uber Zerf\"allungszahlen, J. Reine Angew. Math. {\bf 179} (1938), 97--128; \MR{1581588}
\bibitem{Yao-2024}
O.~X.~M. Yao, Proofs of some conjectures of Andrews and Paule on 2-elongated plane partitions, J. Math. Res. Appl. {\bf 44} (2024), no.~6, 735--740; \MR{4822353}












\end{thebibliography}
\end{document}